\documentclass[a4paper]{amsart}

\usepackage{epic, eepic, amsfonts, latexsym, amssymb, graphicx,
multicol, mathrsfs, color, amscd, verbatim, paralist,
xspace, url, euscript, stmaryrd,  amsmath, enumitem,
bbold, multirow, tikz, mathtools}
\usepackage[all,pdf,cmtip]{xy}

\usepackage[colorlinks, linkcolor=blue, citecolor=magenta, urlcolor=cyan]{hyperref}

\def\hat{\widehat}
\renewcommand\bar{\overline}

  \newcommand\cD{\mathcal{D}}  \newcommand\cL{\mathcal{L}} \newcommand\cO{\mathcal{O}}

\def\fm{{\mathfrak m}}
\def\fp{{\mathfrak p}}

\def\CC{{\mathbb C}}

\def\ZZ{{\mathbb Z}}

\def\PP{{\mathbb P}}

\def\ra{\rightarrow}
\newcommand{\iso}{\stackrel{\sim}{\rightarrow}}

\def\oh{\mathcal{O}}
\def\hat{\widehat}

\def\red{\mathop{\rm red}\nolimits}

\def\GL{\mathop{\rm GL}\nolimits}
\def\SL{\mathop{\rm SL}\nolimits}

\def\grad{\mathop{\rm grad}\nolimits}

\def\Hom{\mathop{\rm Hom}\nolimits}

\def\Cat{\mathop{\rm Cat}\nolimits}
\def\s{\mathop{\rm s}\nolimits}
\def\ss{\mathop{\rm ss}\nolimits}
\def\ps{\mathop{\rm ps}\nolimits}

\def\Res{\mathop{\rm Res}\nolimits}

\def\id{\mathop{\rm id}\nolimits}
\def\Soc{\mathop{\rm Soc}\nolimits}

\def\Jac{\mathop{\rm Jac}\nolimits}
\def\jac{\mathop{\rm jac}\nolimits}
\def\Spec{\mathop{\rm Spec}\nolimits}
\def\Proj{\mathop{\rm Proj}\nolimits}

\def\Gr{\mathop{\rm Gr}\nolimits}
\def\Sym{\mathop{\rm Sym}\nolimits}

\def\wt{\mathop{\rm wt}\nolimits}
\def\codim{\mathop{\rm codim}\nolimits}

\def\gitq{/\hspace{-0.1cm}/}

\def\hess{\mathop{\rm hess}\nolimits}
\def\Hess{\mathop{\rm Hess}\nolimits}

\newcommand{\bA}{\mathbf{A}}
\newcommand{\bB}{\mathbf{B}}
\newcommand{\bbf}{\mathbf{f}}

\newcommand\co{\colon} 

\newtheorem{theorem}{THEOREM}[section]

\newtheorem{proposition}[theorem]{Proposition}
\newtheorem{conjecture}[theorem]{Conjecture}

\theoremstyle{definition}
\newtheorem{example}[theorem]{Example}
\newtheorem{ques}[theorem]{Question}

\newtheorem{prop}[theorem]{Proposition}
\newtheorem{cor}[theorem]{Corollary}

\newtheorem{lemma}[theorem]{Lemma}

\theoremstyle{remark}
\newtheorem{remark}[theorem]{Remark}

\makeatletter
\def\blfootnote{\xdef\@thefnmark{}\@footnotetext}
\makeatother

\begin{document}

\title[Associated Forms: The Binary Case]{Associated Forms and Hypersurface
\vspace{0.1cm}\\
Singularities: The Binary Case}
\author[Alper]{Jarod Alper}
\author[Isaev]{Alexander Isaev}

\address[Alper]{Mathematical Sciences Institute\\
Australian National University\\
Acton, ACT 2601, Australia}
\email{jarod.alper@anu.edu.au}

\address[Isaev]{Mathematical Sciences Institute\\
Australian National University\\
Acton, ACT 2601, Australia}
\email{alexander.isaev@anu.edu.au}

\maketitle

\thispagestyle{empty}

\pagestyle{myheadings}

\begin{abstract}

In the recent articles \cite{EI} and \cite{ai}, it was conjectured that all rational $\GL_n$-invariant functions of forms of degree $d\ge 3$ on $\CC^n$ can be extracted, in a canonical way, from those of forms of degree $n(d-2)$ by means of assigning to every form with nonvanishing discriminant the so-called associated form. The conjecture was confirmed in {\rm \cite{EI}} for binary forms of degree $d \le 6$ as well as for ternary cubics. Furthermore, a weaker version of it was settled in \cite{ai} for arbitrary $n$ and $d$. In the present paper, we focus on the case $n=2$ and establish the conjecture, in a rather explicit way, for binary forms of an arbitrary degree.
\end{abstract}

\section{Introduction}\label{intro}
\setcounter{equation}{0}

This paper concerns a new development in classical invariant theory that originated in the recent article \cite{EI} and was further explored in \cite{ai}, \cite{AIK}. Let $V$ be a complex vector space of dimension $n$ and
$\cO(V)_{d}$ the vector space of forms of degree $d$ on $V$. We assume that $n\ge 2$ and $d\ge 3$. Let $\cO(V)_{d, \Delta}$ be the affine open subvariety of $\cO(V)_d$ consisting of forms for which the discriminant $\Delta$ does not vanish. Upon choosing coordinates, $V$ is identified with $\CC^n$ and we may regard a form $f\in \cO(V)_{d}$ as a polynomial in $x_1, \ldots, x_n$ of degree $d$. Consider the Milnor algebra $M(f) := \CC[x_1,\dots,x_n]/(f_{x_{{}_1}}, \ldots, f_{x_{{}_n}})$ of the isolated singularity at the origin of the hypersurface in $V$ defined by $f\in  \cO(V)_{d,\Delta}$ and let $\fm \subseteq M(f)$ be the maximal ideal. One can then introduce a form on the $n$-dimensional quotient $\fm/\fm^2$ with values in the one-dimensional socle $\Soc(M(f))$ of $M(f)$ as follows:
$$
\fm/\fm^2 	 \to \Soc(M(f)), \quad z  \mapsto \hat z^{\,n(d-2)},
$$
where $\hat z$ is any element of ${\mathfrak m}$ that projects to $z\in{\mathfrak m}/{\mathfrak m}^2$.  There is a canonical isomorphism ${\mathfrak m}/{\mathfrak m}^2\cong V^{\vee}$ and, since the Hessian of $f$ generates the socle, there is also a canonical isomorphism $\Soc(M(f)) \cong \CC$. Hence, one obtains a form $A(f)$ of degree $n(d-2)$ on  $V^{\vee}$, which is called the {\it associated form}\, of $f$ (see Section \ref{setup} for more details on this definition). This form is very natural and in fact has alternative interesting descriptions:  $A(f)$ is a Macaulay inverse system for the Milnor algebra $M(f)$ (see Proposition \ref{P:inverse-system}) and is also a scalar multiple of the residue mapping
$$
\oh(V)_{n(d-2)}  	\to \CC , \quad
g					 \mapsto \Res\left[\begin{array}{c} g\, dx_1\wedge\cdots\wedge dx_n\\ f_{x_{{}_1}}, \dots, f_{x_{{}_n}}\end{array}\right]
$$
(see~\cite[Chap.~III, \S 9]{har}) upon identifying $\oh(V^{\vee})_{n(d-2)}$ and $\oh(V)_{n(d-2)}^{\vee}$ via the polar pairing.

The main object of our study is the morphism
$$A\co  \cO(V)_{d,\Delta} \to \cO(V^{\vee})_{n(d-2)}, \quad f \mapsto A(f)$$
of affine varieties. Upon the identification $V \cong V^{\vee} \cong \CC^n$, the map $A$ is a morphism from the affine variety of nondegenerate forms of degree $d$ on $\CC^n$ to the affine space of forms of degree $n(d-2)$ on $\CC^n$. As first observed in \cite{EI}, for certain values of $n$ and $d$ one can recover all $\GL_n$-invariant rational functions on forms of degree $d$ from those on forms of degree $n(d-2)$ by evaluating the latter on associated forms, i.e., by composing them with $A$. Motivated by the above fact, in \cite{ai} we proposed a conjecture asserting that an analogous statement holds for all $n$ and $d$.  Precisely, the conjecture stated in \cite{ai} is:

\begin{conjecture}\label{C:main} For every regular $\GL(V)$-invariant function $I$ on $\cO(V)_{d, \Delta}$ there exists a rational\, $\GL(V^{\vee})$-invariant function $J$ on $\cO(V^{\vee})_{n(d-2)}$ defined on the image of $A$ such that $I=J \circ A$.
\end{conjecture} 

In other words, the conjecture states that the invariant theory of forms of degree $d$ can be extracted, in a canonical way, from that of forms of degree $n(d-2)$ at least at the level of rational invariant functions. In \cite{EI}, Conjecture \ref{C:main} was shown to hold for binary forms of degrees $3\le d\le 6$ as well as for ternary cubics, and in \cite{ai} a weaker variant was established for arbitrary $n$ and $d$. 

As explained in \cite{EI} and \cite{ai}, the original motivation for studying associated forms came from singularity theory, namely, the well-known Mather-Yau theorem. This theorem states that two isolated hypersurface singularities in $\CC^n$ are biholomorphically equivalent if and only if their Tjurina algebras are isomorphic. The proof of this theorem is not constructive, and finding an effective way of recovering a hypersurface germ from its Tjurina algebra is a long-standing open question called the reconstruction problem. If Conjecture \ref{C:main} is settled, it will provide a method for extracting a complete system of biholomorphic invariants of homogeneous hypersurface singularities from their Milnor algebras (which coincide with their Tjurina algebras). Such a system of invariants can be regarded as a solution to the reconstruction problem in this case, and settling the reconstruction problem was our original goal.  

On the other hand,  Conjecture \ref{C:main} is rather interesting from the purely invariant-theoretic viewpoint
and surprisingly enlightening even in the case of binary forms.
 In the main result of the present paper, we settle a stronger version of the conjecture for $n=2$ as follows: 

\begin{theorem}\label{T:binary} The morphism $A$ maps $\cO(\CC^2)_{d, \Delta}$ to $\oh(\CC^2)_{2(d-2), \Cat}$, and for every regular $\GL_2$-invariant function $I$ on $\cO(\CC^2)_{d, \Delta}$ there exists a regular $\GL_2$-invariant function $J$ on $\oh(\CC^2)_{2(d-2), \Cat}$ such that $I = J \circ A$, where $\Cat$ denotes the catalecticant invariant.\end{theorem}
\noindent Upon showing that $A$ sends $\cO(\CC^2)_{d, \Delta}$ to $\oh(\CC^2)_{2(d-2), \Cat}$, this theorem is equivalent to the statement that the induced map
$$\bar{A} \co \cO(\CC^2)_{d, \Delta} \gitq \GL_2 \to \cO(\CC^2)_{2(d-2), \Cat} \gitq \GL_2$$
of affine good GIT quotients is a closed immersion.

\medskip 

\noindent {\bf Sketch of the proof of Theorem \ref{T:binary}.} 
In general, the projectivization of the morphism $A$, which assigns to every nondegenerate form $f$ the associated form $A(f)$, factors as 
$$A: \PP(\oh(V)_d)_{\Delta} \xrightarrow{\nabla} \Gr(n, \oh(V)_{d-1})_{\Res} \xrightarrow{\bA} \PP(\cO(V^{\vee})_{n(d-2)}),
$$
where $\bA$ is the morphism that assigns to every $n$-dimensional subspace $W \subseteq \oh(V)_{d-1}$ with nonvanishing resultant a form $\bA(W)$ defined analogously to the associated form, and where $\nabla$ is the morphism that assigns to a form the subspace in $ \oh(V)_{d-1}$ generated by its first-order partial derivatives.  In Section \ref{setup}, we discuss details of this factorization and in Section \ref{GITinterpr} we interpret Conjecture \ref{C:main} in terms of the GIT quotients of $\bA$ and $\nabla$.

To prove Theorem \ref{T:binary}, we proceed by showing that:
\begin{enumerate}
	\item the morphism $\nabla$ preserves semistability,
	\item the induced morphism $\bar{\nabla}\co\PP(\oh(\CC^2)_d)^{\ss} \gitq \SL_2 \to \Gr(2, \oh(\CC^2)_{2(d-2)})^{\ss} \gitq \SL_2$ of GIT quotients is a closed immersion, and
	\item the morphism $\bA \co \Gr(2, \oh(\CC^2)_{d-1})_{\Res} \to \PP(\oh(\CC^2)_{2(d-2)})_{\Cat}$ is an isomorphism. \end{enumerate}

Step (1) is established in Proposition \ref{P:gradient-ss}.  It is not hard to see that (1) implies that
$\bar{\nabla}$ is a finite morphism. In fact, we show that $\bar{\nabla}$ is finite and injective (Corollary \ref{C:finite-injective}). Therefore, to establish (2), it suffices to prove that the image of $\bar{\nabla}$ is normal, which we obtain in Corollary \ref{C:normal}.  Step (3) is established in Proposition \ref{P:A-iso}.

{\bf Acknowledgements.} We are grateful to Daniel Erman for his valuable suggestions, which were instrumental for the proof of Proposition \ref{P:normal}. We also thank Felix Janda for explaining to us the connection between associated forms and the residue symbol, as well as Maksym Fedorchuk for comments on an earlier draft of the paper. Special thanks go to the referee who put a tremendous amount of effort into understanding our results and made extensive recommendations, which simplified many of our arguments and greatly improved the exposition of this paper. This work is supported by the Australian Research Council.

\section{The associated form: definition and properties}\label{setup}
\setcounter{equation}{0}

\subsection{The associated form of a nondegenerate form}  \label{SS:assoc-form}
Let $V$ be a complex vector space of dimension $n$. Denote by $\cO(V)$ the symmetric algebra $\Sym(V^{\vee})$ on $V^{\vee}$ and by $\cO(V)_j:= \Sym^j(V^{\vee})$ the vector space of forms on $V$ of degree $j$. One has $\cO(V)=\oplus_{j=0}^{\infty} \cO(V)_j$. If we choose an isomorphism $V \cong \CC^n$, we may identify $\cO(V)$ with the algebra of polynomials $\CC[x_1,\dots,x_n]$ in $x_1, \ldots, x_n$ and $\cO(V)_j$ with the vector space $\CC[x_1, \ldots, x_n]_j$ of homogeneous polynomials of degree $j$. We assume throughout that $n \ge 2$.

Fix $d\ge 3$. Recall that the {\it discriminant}\, $\Delta$ is a form on $\oh(V)_d$ with the property that $\Delta(f) \neq 0$ if and only if the zero locus of $f$ has an isolated singularity at $0$.  Let $\cO(V)_{d, \Delta}$ be the affine open subvariety of $\cO(V)_d$ of forms with nonzero discriminant.  For $f \in \cO(V)_{d, \Delta}$, we define the Milnor algebra of the corresponding singularity to be
$$M(f) := \oh(V) / J(f), $$ 
where $J(f)$ is the Jacobian ideal, i.e., the ideal generated by all first-order partial derivatives of $f$. Choosing coordinates $x_1, \ldots, x_n$ in $V$, we may write the Milnor algebra as $M(f) = \CC[x_1,\dots,x_n]/(f_{x_{{}_1}}, \ldots, f_{x_{{}_n}})$. It is well-known that 
$M(f)$ is a standard graded local Artinian Gorenstein algebra whose socle $\Soc(M(f)) = M(f)_{n(d-2)}$ is generated in degree $n(d-2)$ by the image $\bar{\hess(f)} \in M(f)$ of the Hessian $\hess(f):= \det \Hess(f)$, where $\Hess(f)$ is the Hessian matrix $\big({\partial^2 f}/{\partial x_i \partial x_j} \big)_{i,j}$ (cf.~Lemma \ref{fourconds} below).

We will now introduce the {\it associated form}\, $A(f) \in \cO(V^{\vee})_{n(d-2)}$ of $f$. Upon identifying $\cO(V^{\vee})_{n(d-2)}$ with $\CC[y_1, \ldots, y_n]_{n(d-2)}$, where the coordinates $y_i$ are dual to $x_i$, we define $A(f)$ by the following formula:
$$(y_1 \bar{x}_1 + y_2 \bar{x}_2 + \cdots + y_n \bar{x}_n)^{n(d-2)} = A(f)(y_1, \ldots, y_n) \cdot \bar{\hess(f)} \in M(f),$$
with $\bar{x}_i \in M(f)$ being the image of the coordinate function $x_i$. It is not hard to see that the induced map
\begin{equation} \label{E:assoc}
A \co \cO(V)_{d, \Delta} \to \cO(V^{\vee})_{n(d-2)}, \quad f \mapsto A(f)
\end{equation}
is a morphism of affine varieties. 

\begin{example}  \label{E:example1} If $f = a_1 x_1^d + \cdots + a_n x_n^d$ for nonzero $a_i \in \CC$, then one computes $\hess(f) = (a_1 \cdots a_n)(d(d-1))^n (x_1 \cdots x_n)^{d-2}$ and
$$A(f)(y_1, \ldots, y_n) = \frac{1}{a_1 \cdots a_n} \frac{(n(d-2))!}{(d!)^n} (y_1 \cdots y_n)^{d-2}.$$
\end{example}

\smallskip

Next, there are natural actions of $\GL(V)$ on $\cO(V)_{d}$ and $\cO(V^{\vee})_{n(d-2)}$ as follows: if $g \in \GL(V)$, then  $(gf)(x): = f (x \cdot g^{-t})$ for $f \in \cO(V)_{d}$ and $x=(x_1, \ldots, x_n) \in V$, and $(gF)(y): = F(y \cdot g)$ for $F \in \cO(V^{\vee})_{n(d-2)}$ and $y = (y_1, \ldots, y_n) \in V^{\vee}$.  The affine open subvariety $\oh(V)_{d, \Delta}$ of $\oh(V)_d$ is clearly $\GL(V)$-invariant. The following lemma gives an equivariance property for the morphism $A$:

\begin{lemma}  
\label{L:equiv1} \it
For $f \in \cO(V)_{d, \Delta}$ and $g \in \GL(V)$, we have
\begin{equation} \label{E:equiv1}
A(g f) = (\det g)^2 \cdot g A(f).
\end{equation}
\end{lemma}

\begin{proof}
There is an isomorphism of Milnor algebras $\varphi \co M(f) \iso M(gf)$ defined by $\bar{x_i} \mapsto \bar{g x_i}$. For $\bar{x} = (\bar{x}_1, \ldots, \bar{x}_n)$ and $y = (y_1, \ldots, y_n)$, we have by definition the equality $(y \cdot \bar{x}^t)^{n(d-2)} = A(f)(y) \cdot \bar{\hess(f)}$ in $M(f)$.  By noting that $\bar{gx} = \bar{x} \cdot g^{-t}$, we compute
\begin{equation} \label{E:equiv-left}
\varphi\big( (y \cdot \bar{x}^t)^{n(d-2)}\big) = \big(y \cdot ( \bar{gx})^t\big)^{n(d-2)} = \big(y \cdot g^{-1} \cdot \bar{x}^t \big)^{n(d-2)}.
\end{equation}
On the other hand, using the fact that $g \hess(f) = (\det g)^2 \cdot \hess(g f)$, we have
\begin{equation} \label{E:equiv-right}
\varphi(A(f)(y) \cdot \bar{\hess(f)}) = A(f)(y) \cdot \bar{g \hess(f)} = A(f)(y) \cdot (\det g)^2 \cdot \bar{\hess(gf)}.
\end{equation}
Comparing \eqref{E:equiv-left} and \eqref{E:equiv-right}, we obtain $A(gf)(y \cdot g^{-1}) = (\det g)^2 \cdot A(f)(y) $.
\end{proof}

\begin{remark}  Lemma \ref{L:equiv1} was established in \cite[Prop.~2.1]{ai} but due to its importance and the fact that we will generalize it in Lemma \ref{L:equiv2} below, we included the proof.  
\end{remark}

\subsection{The associated form of a finite morphism and a factorization of $A$}\label{assocformfin}  As before, let $V$ be a complex vector space of dimension $n \ge 2$ and let $d \ge 3$. We will now generalize the above construction from nondegenerate forms $f \co V \to \CC$ of degree $d$ to finite morphisms $\bbf = (f_1, \ldots, f_n) \co V \to V$ defined by $n$ forms of degree $d-1$. 

Let $\oh(V)_{d-1}^{\oplus n}$ be the vector space of $n$-tuples $\bbf = (f_1, \ldots, f_n)$ of forms of degree $d-1$.  Recall from \cite[Ch.~13]{GKZ} that the {\it resultant}\, $\Res$ on $\oh(V)_{d-1}^{\oplus n}$ is a form with the property that $\Res(\bbf) \neq 0$ if and only if $f_1, \ldots, f_n$ have no common zeroes away from $0$.   For $\bbf=(f_1, \ldots, f_n)$, introduce the algebra
$$
M(\bbf) := \oh(V) / (f_1, \ldots, f_n).
$$
Fixing coordinates in $V$, we have a well-known lemma:

\begin{lemma}\label{fourconds} \it
If\, $\bbf = (f_1, \ldots, f_n) \in (\oh(V)_{d-1})^{\oplus n}$, then the following are equivalent:
\begin{enumerate}
	\item[\rm (1)] the resultant $\Res(\bbf)$ is nonzero;
	\item[\rm (2)] the algebra $M(\bbf)$ has finite vector space dimension;
	\item[\rm (3)] the morphism $\bbf \co V \to V$ is finite;
	\item[\rm (4)] the $n$-tuple $\bbf$ is a homogeneous system of parameters of $\oh(V)$, i.e., the Krull dimension of $M(\bbf)$ is $0$.
\end{enumerate}
If the above conditions are satisfied, then $M(\bbf)$ is a standard graded local Artinian Gorenstein algebra whose socle $\Soc(M(\bbf)) = M(\bbf)_{n(d-2)}$ is generated in degree $n(d-2)$ by the image $\bar{\jac(\bbf)} \in M(\bbf)$ of the Jacobian $\jac(\bbf):= \det \Jac(\bbf)$, where $\Jac(\bbf)$ is the Jacobian matrix $\big({\partial f_i}/{\partial x_j} \big)_{i,j}$.\end{lemma}

\begin{proof} The implication ${\rm (1)}\Rightarrow{\rm (2)}$ follows from the Nullstellensatz. Next, ${\rm (2)}$ yields that there exists a positive integer $N$ such that $\oh(V)_j$ lies in the ideal $(f_1,\dots,f_n)\subseteq\oh(V)$ for all $j>N$. Therefore, every $f\in \oh(V)$ can be written as a linear combination of monomials of degree not exceeding $N$ with coefficients that are polynomials in $f_1,\dots,f_n$. This means that the induced ring homomorphism $\bbf^*\co\oh(V)\to \oh(V)$ is finite, which establishes ${\rm (3)}$. Each of the implications ${\rm (3)} \Rightarrow{\rm (4)} \Rightarrow {\rm (1)}$ is straightforward. The conditions (1)--(4) imply that $M(\bbf)$ is a local complete intersection, hence a local Artinian Gorenstein algebra. The last statement of the lemma follows, e.g., from \cite[Lem.~3.4]{Sa}. \end{proof}

\begin{remark} \label{R:hilbert-function}
As we pointed out in Lemma \ref{fourconds}, the algebra $M(\bbf)$ has a natural standard grading: $M(\bbf) = \bigoplus_{i=0}^{\infty} M(\bbf)_i$.  It is well-known (cf.~\cite[Cor.~3.3]{St}) that the Hilbert function
$H(t):=\sum_{i=0}^{\infty}\dim{M}({\mathbf f})_i\,t^i$
of ${M}({\mathbf f})$ is given by
$ \label{E:hilbert-function}
H(t)=(t^{d-2}+\dots+t+1)^n.
$
We note that  $\dim M(\bbf)_{i} =  \dim M(\bbf)_{n(d-2)-i}$ for $0 \le i \le n(d-2)$ and that $\dim M(\bbf)_{i} = 0$ for $i > n(d-2)$.  In particular, as already stated in Lemma \ref{fourconds}, the socle $M(\bbf)_{n(d-2)}$ is one-dimensional.
\end{remark}

We let $( \oh(V)_{d-1}^{\oplus n})_{\Res}$ be the affine open subvariety of $\oh(V)_{d-1}^{\oplus n}$ consisting of $n$-tuples of forms with nonzero resultant. Upon choosing coordinates $x_1, \ldots, x_n$ in $V$ and the dual coordinates $y_1, \ldots, y_n$ in $V^{\vee}$, we may define the {\it associated form}\,\linebreak $\bA(\bbf) \in \cO(V^{\vee})_{n(d-2)}$ of $\bbf \in ( \oh(V)_{d-1}^{\oplus n})_{\Res}$ by the following formula:
$$(y_1 \bar{x}_1 + y_2 \bar{x}_2 + \cdots + y_n \bar{x}_n)^{n(d-2)} = \bA(\bbf)(y_1, \ldots, y_n) \cdot \bar{\jac(\bbf)} \in M(\bbf),$$
where $\bar{x}_i \in M(\bbf)$ is the image of $x_i$. It is again not hard to see that the induced map
\begin{equation} \label{E:assoc2}
\bA \co ( \oh(V)_{d-1}^{\oplus n})_{\Res} \to \cO(V^{\vee})_{n(d-2)}, \quad \bbf \mapsto \bA(\bbf)
\end{equation}
is a morphism of affine varieties.  

We point out that the construction of the associated form $\bA(\bbf)$ of an $n$-tuple\linebreak $\bbf \in ( \oh(V)_{d-1}^{\oplus n})_{\Res}$ generalizes that of the associated form $A(f)$ of a form\linebreak $f \in \oh(V)_{d, \Delta}$.  Namely, if we consider the gradient morphism 
$$\grad \co \oh(V)_{d} \to \oh(V)_{d-1}^{\oplus n}, \quad f \mapsto (f_{x_{{}_1}}, \ldots, f_{x_{{}_n}}),$$
then $\Delta(f) = 0$ if and only if $\Res(\grad(f)) = 0$ (the induced morphism $\oh(V)_{d, \Delta} \to (\oh(V)_{d-1}^{\oplus n})_{\Res}$ will also be referred to as $\grad$). Further, for any $f \in \oh(V)_{d, \Delta}$, we have $\hess(f) = \jac(\grad(f))$. Therefore, the following holds:
\begin{lemma} \label{L:factorization} \it The morphism $A$ factors as
\begin{equation}
A: \oh(V)_{d, \Delta} \xrightarrow{\grad} (\oh(V)_{d-1}^{\oplus n})_{\Res} \xrightarrow{\bA} \cO(V^{\vee})_{n(d-2)}.
\label{factoriz1}
\end{equation}
\end{lemma}

The vector space $ \oh(V)_{d-1}^{\oplus n}$ has a natural action of $\GL(V) \times \GL_n$ via
\begin{equation}
((g_1, g_2)  {\mathbf f}) (x) := {\mathbf f} (x \cdot g_1^{-t}) \cdot g_2^{-1}\label{doubleaction}
\end{equation}
for $g_1\in\GL(V)$, $g_2 \in \GL_n$ and ${\mathbf f}=(f_1,\dots,f_n) \in  \oh(V)_{d-1}^{\oplus n}$.
Clearly, the open affine subvariety $(\oh(V)_{d-1}^{\oplus n})_{\Res}$ is $\GL(V) \times \GL_n$-invariant.

\begin{lemma}\label{L:equiv2} \it For every ${\mathbf f}\in (\oh(V)_{d-1}^{\oplus n})_{\Res}$, $g_1\in\GL(V)$, $g_2 \in \GL_n$, one has
\begin{equation} \label{E:equiv2}
\displaystyle{\mathbf A}((g_1,g_2) \bbf)=\det(g_1 g_2) \cdot g_1 {\mathbf A}({\mathbf f}).
\end{equation}
\end{lemma} 

\begin{proof}
The argument of Lemma \ref{L:equiv1} implies $\bA((g_1,\id) \bbf)=(\det g_1) \cdot g_1  \bA(\bbf)$.  On the other hand, since the ideal generated by the forms $f_1, \ldots, f_n$ is equal to the ideal generated by the $n$ forms of $\bbf \cdot g_2^{-1}$ and since $\jac(\bbf \cdot g_2^{-1}) = (\det g_2)^{-1} \cdot\jac(\bbf)$, we see that $\bA((\id,g_2) \bbf)=(\det g_2) \cdot  \bA(\bbf)$.
\end{proof}

\begin{remark} The gradient morphism
$\grad \co \oh(V)_d \to \oh(V)_{d-1}^{\oplus n}$
is equivariant with respect to the diagonal inclusion $\GL(V) \hookrightarrow \GL(V) \times \GL_n$, $g \mapsto (g, g)$, that is, $\grad(g  f) = (g, g)  \grad(f)$ for $g \in \GL(V)$ and $f \in \oh(V)_d$ (recall here that we work with the fixed coordinates $x_1,\dots,x_n$ in $V$ hence with a fixed isomorphism $\GL(V)\cong \GL_n$). Therefore, equivariance formula \eqref{E:equiv2} induces equivariance formula \eqref{E:equiv1} via the gradient morphism, that is, factorization (\ref{factoriz1}) is equivariant with respect to the given actions.
\end{remark}

\subsection{Projectivization of the morphisms $A$ and $\bA$} 
The constructions of Sections \ref{SS:assoc-form} and \ref{assocformfin} can be projectivized. Let $\PP(\cO(V)_d)$ be the projectivization of $\cO(V)_d$ and $\PP(\cO(V)_d)_{\Delta}$ the affine open subvariety where the discriminant $\Delta$ does not vanish. Then morphism \eqref{E:assoc} gives rise to a morphism
$$A \co \PP(\cO(V)_d)_{\Delta} \to \PP(\cO(V^{\vee})_{n(d-2)}),$$
which is $\SL(V)$-equivariant.
{\it We abuse notation by using the same symbol $A$ to denote morphism \eqref{E:assoc} as well as the induced projectivized morphism.}

Next, denote by $\Gr(n, \cO(V)_{d-1})$ the Grassmannian of $n$-dimensional subspaces of $\cO(V)_{d-1}$. Clearly, $\Gr(n, \cO(V)_{d-1})$ is a $\GL_n$-quotient of the open subvariety of $\oh(V)_{d-1}^{\oplus n}$ consisting of $n$-tuples $\bbf = (f_1, \ldots, f_n)$ where $f_1, \ldots, f_n$ are linearly independent.  The resultant $\Res$ on $\oh(V)_{d-1}^{\oplus n}$ descends to a section, also denoted by $\Res$, of a power of the very ample generator of the Picard group of $\Gr(n, \oh(V)_{d-1})$.  Let $\Gr(n, \cO(V)_{d-1})_{\Res}$ be the affine open subvariety where $\Res$ does not vanish and consider the $\GL(V)$-equivariant morphism 
$$
( \oh(V)_{d-1}^{\oplus n})_{\Res} \to \Gr(n, \cO(V)_{d-1})_{\Res}, \quad  \bbf = (f_1, \ldots, f_n) \mapsto \langle f_1, \ldots, f_n \rangle,
$$
which is a $\GL_n$-torsor. Then, by equivariance property (\ref{E:equiv2}), the morphism $\bA$ from \eqref{E:assoc2} composed with the projection $\cO(V^{\vee})_{n(d-2)}\setminus\{0\}\to\PP(\cO(V^{\vee})_{n(d-2)})$ factors as
$$
( \oh(V)_{d-1}^{\oplus n})_{\Res} \to \Gr(n, \cO(V)_{d-1})_{\Res} \xrightarrow{\bA} \PP(\cO(V^{\vee})_{n(d-2)}).
$$
{\it Note that we abuse notation by using the symbol $\bA$ to refer to both the morphism 
$( \oh(V)_{d-1}^{\oplus n})_{\Res} \to \cO(V^{\vee})_{n(d-2)}$ and $\Gr(n, \cO(V)_{d-1})_{\Res} \to \PP(\cO(V^{\vee})_{n(d-2)})$.  In each occurrence of the symbols $A$ and $\bA$, it will be clear which morphism we refer to.}

Further, the gradient morphism descends to an $\SL(V)$-equivariant morphism
$$\begin{aligned}
\nabla \co \PP(\oh(V)_{d}) \setminus \{ f \, \mid \, f_{x_{{}_1}}, \ldots, f_{x_{{}_n}} \text{are linearly dependent} \} & \to  \Gr(n, \oh(V)_{d-1}),  \\
f &\mapsto \langle f_{x_{{}_1}}, \ldots, f_{x_{{}_n}} \rangle,
\end{aligned}$$
and we obtain a factorization analogous to (\ref{factoriz1}):
\begin{equation} \label{E:factorization-proj}
A: \PP(\oh(V)_d)_{\Delta} \xrightarrow{\nabla} \Gr(n, \oh(V)_{d-1})_{\Res} \xrightarrow{\bA} \PP(\cO(V^{\vee})_{n(d-2)}).
\end{equation}
From (\ref{E:equiv1}) and (\ref{E:equiv2}), it is clear that factorization (\ref{E:factorization-proj}) is $\SL(V)$-equivariant with respect to the natural actions.

\subsection{The polar pairing and Macaulay inverse systems}

Let $x_1, \ldots, x_n$ be coordinates in $V$ and $y_1, \ldots, y_n$ the dual coordinates in $V^{\vee}$. Recall that the algebra $\cO(V^{\vee})$ is an $\cO(V)$-module via differentiation:
$$(h \circ F) (y_1, \ldots, y_n) := h\left(\frac{\partial}{\partial y_1}, \ldots, \frac{\partial}{\partial y_n}\right)F(y_1, \ldots, y_n),$$
where $h \in \oh(V)$ and $F \in \oh(V^{\vee})$.  For a positive integer $j$, differentiation induces a perfect pairing
$$ \cO(V)_j \times \cO(V^{\vee})_j \to \CC, \quad (h, F) \mapsto h \circ F,$$
which is independent of the choice of coordinates and is often referred to as the {\it polar pairing}.  For $F \in \cO(V^{\vee})_j$, we now introduce the homogenous ideal
$$F^{\perp} := \{h\in \oh(V)\, \mid \, h \circ F = 0 \} \subseteq \oh(V),$$
which is clearly independent of scaling and thus is well-defined for $F \in \PP(\cO(V^{\vee})_j)$.
It is well-known that the quotient ring $\oh(V) / F^{\perp}$ is a local standard graded Artinian Gorenstein algebra of socle degree $j$. In particular, $\dim (\oh(V) / F^{\perp} )_{j} = 1$, $\dim (\oh(V) / F^{\perp} )_{i} = 0$ for $i > j$, and the symmetry relation $\dim (\oh(V) / F^{\perp} )_{i} = \dim (\oh(V) / F^{\perp} )_{j-i}$ holds for all $0\le i\le j$.

We recall the following well-known proposition (cf.~\cite[Lem.~2.12]{IK}):

\begin{proposition} \label{prop-correspondence}
The correspondence $F \mapsto \oh(V)/F^{\perp}$ induces a bijection
$$
\PP(\cO(V^{\vee})_j)  \to
\left\{ 
 	\begin{array}{l} 
		\text{local Artinian Gorenstein quotient algebras $\oh(V)/I$}\\ 
		\text{of socle degree $j$, where the ideal $I$ is homogeneous}\\
		 \end{array} \right\}.
$$
\end{proposition}

\begin{remark}\label{invsysrem}
\noindent Given a homogenous ideal $I \subseteq \oh(V)$ such that $\oh(V)/I$ is an Artinian Gorenstein algebra of socle degree $j$, Proposition \ref{prop-correspondence} implies that there is a  form $F \in  \cO(V^{\vee})_j$, unique up to scaling, such that $I = F^{\perp}$. In fact, the uniqueness part of this statement can be strengthened: if $I\subseteq F^{\perp}$, then $I = F^{\perp}$ and all forms with this property are mutually proportional. Indeed, $I\subseteq F^{\perp}$ implies $I_j\subseteq F^{\perp}$, and the claim follows from the fact that $I_j$ has codimension 1 in $\oh(V)_j$. Any such form $F$ is called {\it a {\rm (}homogeneous{\rm )} Macaulay inverse system for $\oh(V)/I$}, and its image in $\PP(\cO(V^{\vee})_j)$ is called {\it the {\rm (}homogeneous{\rm )} Macaulay inverse system for $\oh(V)/I$}.
\end{remark}

The following proposition gives an alternative interpretation of the associated form in terms of Macaulay inverse systems.  This proposition was established in \cite[Prop.~3.2]{ai} in greater generality but we include a direct proof below in the case at hand for the reader's convenience.

\begin{prop} \label{P:inverse-system} \it
If $\bbf = (f_1, \ldots, f_n) \in (\oh(V)_{d-1}^{\oplus n})_{\Res}$, 
then $\bA(\bbf) \in \oh(V^{\vee})_{n(d-2)}$ is a Macaulay inverse system for the algebra $M(\bbf)$.
In particular, if $f \in \cO(V)_{d,\Delta}$, then 
$A(f)  \in \oh(V^{\vee})_{n(d-2)}$ is a Macaulay inverse system for the Milnor algebra $M(f)$.
\end{prop}

\begin{proof} For $\bar{x} = (\bar{x}_1, \ldots, \bar{x}_n)$ and $y = (y_1, \ldots, y_n)$ 
we have by definition the equality $(y \cdot \bar{x}^t)^{n(d-2)} = \bA(\bbf)(y) \cdot \bar{\jac(\bbf)}$ in $M(\bbf)$.  If $j_1, \ldots, j_n$ are nonnegative integers with sum $j$, then a simple calculation yields
$$\left(\frac{\partial^j}{\partial y_1^{j_1} \cdots \partial y_n^{j_n}} \bA(\bbf) \right) \bar{\jac(\bbf)}=  \frac{ (n(d-2))!}{(n(d-2) -j)!}  (y \cdot \bar{x}^t)^{n(d-2) - j} \, \bar{x}_1^{j_1} \cdots \bar{x}_n^{j_n}.$$
Therefore, for any $h \in {\mathcal O}(V)_j$ we have 
$$\big( h \circ  \bA(\bbf) \big) \bar{\jac(\bbf)} = \frac{ (n(d-2))!}{(n(d-2) - j)!} (y \cdot \bar{x}^t)^{n(d-2) - j} \, h(\bar{x}_1, \ldots, \bar{x}_n),$$
which is zero if and only if $h(\bar{x}_1, \ldots, \bar{x}_n) = 0$, or equivalently, if and only if $h$ lies in the ideal $(f_1, \ldots, f_n)$.  We have thus established that
$(f_1, \ldots, f_n) = \bA(\bbf)^{\perp}$.
\end{proof}

\begin{remark}
The above proposition allows us to interpret the morphism\linebreak $A \co \PP(\oh(V)_d)_{\Delta} \to \PP(\oh(V^{\vee})_{n(d-2)})$ as taking a form $f$ to the Macaulay inverse system for the Milnor algebra $M(f)$.  Similarly, the morphism $\bA \co \Gr(n, \oh(V)_{d-1})_{\Res} \to  \PP(\oh(V^{\vee})_{n(d-2)})$ takes a subspace $\langle f_1, \ldots, f_n \rangle \in \Gr(n, \oh(V)_{d-1})_{\Res}$ to the Macaulay inverse system for the algebra ${\mathcal O}(V) / (f_1, \ldots, f_n)$.
\end{remark}

\subsection{The image of ${\mathbf A}$}\label{sectimage} Consider the locally closed subvariety $U_{\Res} \subseteq \PP(\oh(V^{\vee})_{n(d-2)})$ of forms $F$ such that $\dim F^{\perp} \cap \oh(V)_{d-1} = n$ and the subspace $F^{\perp} \cap \oh(V)_{d-1}$ lies in $\Gr(n,\oh(V)_{d-1})_{\Res}$. By Proposition \ref{P:inverse-system}, the image of $\bA$ is contained in $U_{\Res}$.

\begin{prop} \label{P:A-immersion} \it The morphism $\bA:\Gr(n, \oh(V)_{d-1})_{\Res}\to U_{\Res}$ is an isomorphism.
\end{prop}

\begin{proof}
The morphism $\bB \co U_{\Res} \to \Gr(n, \oh(V)_{d-1})_{\Res}$ given by $F \mapsto F^{\perp} \cap \oh(V)_{d-1}$ yields the diagram
$$\xymatrix{
\hspace{-2cm}\Gr(n, \oh(V)_{d-1})_{\Res} \ar[r]^{\hspace{1cm} \bA} \ar[rd]^{{\mathrm {id}}}		&	U_{\Res} \ar[d]^{\bB} \\
											& \Gr(n, \oh(V)_{d-1})_{\Res},
}$$		
which is commutative by Proposition \ref{P:inverse-system}. As $\bB$ is separated, it follows immediately that $\bA$ is a closed immersion. Next, if $F \in U_{\Res}$, then for the ideal $I\subseteq{\mathcal O}(V)$ generated by $F^{\perp} \cap \oh(V)_{d-1}$, we have the inclusion $I \subseteq F^{\perp}$. By Remark \ref{invsysrem}, the form $F$ is the inverse system for ${\mathcal O}(V)/I$, and therefore $F = \bA( F^{\perp} \cap \oh(V)_{d-1})$.
\end{proof}

\section{Interpretation of Conjecture \ref{C:main} via GIT quotients}\label{GITinterpr}
\setcounter{equation}{0}

In this section, we interpret Conjecture \ref{C:main} in terms of properties of the GIT quotients of the morphisms $A$ and $\bA$.

\subsection{Review of Geometric Invariant Theory}  \label{S:GIT} We quickly review some notions from geometric invariant theory (GIT); see \cite{git} for more details.

Let $G$ be a reductive algebraic group over $\CC$. If $W$ is a $G$-representation, then the  {\it GIT quotient of $W$ by $G$}\, is the morphism $\pi \co W \to W \gitq G := \Spec \oh(W)^G$.  
The quotient $W \gitq G$ has desirable properties---in particular, it parametrizes closed $G$-orbits in $W$.  There is an induced action of $G$ on the projective space $\PP(W)$, and we define the loci of {\it semistable, polystable}\, and {\it stable}\, points in $\PP(W)$, respectively, by:
$$
\begin{array}{ll}
\PP(W)^{\ss} & \hspace{-0.3cm}:=\left\{w \in \PP(W) \, \middle| \begin{array}{l} \hbox{there exists $I \in \oh(W)_j^G$ for $j > 0$ such that $I(\hat{w}) \neq 0$} \\ \hbox{where $\hat{w} \in W$ is any lift of $w$.} \end{array} \right\},\\
\vspace{-0.1cm}\\
\PP(W)^{\ps}& \hspace{-0.3cm}:=\{w\in \PP(W)^{\ss}\mid \hbox{$G w$ is closed in $\PP(W)^{\ss}$}\},\\
\vspace{-0.1cm}\\
\PP(W)^{\s}& \hspace{-0.3cm}:= \{w\in \PP(W)^{\ss}\mid \hbox{$G w$ is closed in $\PP(W)^{\ss}$ and $\dim G w=\dim G$}\}.
\end{array}
$$
It is a fact that $w \notin \PP(W)^{\ss}$ if and only $0 \in \bar{ G  \hat{w}}$ for a lift $\hat{w} \in W$ of $w$.  It follows that the affine cone over the complement $\PP(W) \setminus \PP(W)^{\ss}$ is simply $\pi^{-1}(\pi(0))$.   

More generally, if $X \subseteq \PP(W)$ is a $G$-invariant projective subvariety, then we can define the loci of semistable, polystable and stable points as $X^{\ss} := X \cap \PP(W)^{\ss}$, $X^{\ps} := X \cap \PP(W)^{\ps}$ and $X^{\s} := X \cap \PP(W)^{\s}$. One has the commutative diagram
$$\xymatrix{
\hspace{0.2cm}X^{\s}\,  \ar@{^(->}[r] \ar[d]		& \hspace{0.3cm}X^{\ss} \ar[d] \\
X^{\s} \gitq G \, \ar@{^(->}[r]			& X^{\ss} \gitq G := \Proj \bigoplus_{j \ge 0} \Gamma(X, \oh(j))^G,
}$$
where
\begin{itemize}
\item $X^{\ss} \to X^{\ss} \gitq G$ is called the {\it GIT quotient of $X^{\ss}$ by $G$} and the points of $X^{\ss} \gitq G$ are in bijective correspondence with the closed $G$-orbits in $X^{\ss}$ (i.e., the $G$-orbits of polystable points);
\item $X^{\s} \to X^{\s} \gitq G$ is called the {\it geometric GIT quotient of $X^{\s}$ by $G$} and the points of $X^{\s} \gitq G$ are in bijective correspondence with  $G$-orbits in $X^{\s}$ (thus $X^{\s} \gitq G$ parametrizes the usual orbit space);
\item $ X^{\ss} \gitq G$ is a projective variety, $X^{\s}\gitq G \subseteq X^{\ss} \gitq G$ is an open subvariety and $X^{\s}$ is the preimage of $X^{\s} \gitq G$ under the quotient morphism $X^{\ss} \to X^{\ss} \gitq G $.
\end{itemize}

If $Y$ is a variety over $\CC$ with an action of $G$, we say that a surjective $G$-invariant morphism $\pi \co Y \to Z$ is a {\it good GIT quotient}\, if $\pi$ is affine and $\oh_Z = (\pi_* \oh_Y)^G$. If a good quotient exists, it is unique up to isomorphism and is often denoted by $Y \gitq G$.  The morphisms $W \to W \gitq G$, $X^{\ss} \to X^{\ss} \gitq G$, and $X^{\s} \to X^{\s} \gitq G$ defined above are all examples of good GIT quotients. 

 If $Y_1 \to Y_1 \gitq G$ and $Y_2 \to Y_2 \gitq G$ are two good GIT quotients, then any morphism $F \co Y_1 \to Y_2$ that sends $G$-orbits to $G$-orbits naturally induces a morphism $\bar{F} \co Y_1 \gitq G \to Y_2 \gitq G$.  We will always denote the quotient morphism by overlining the corresponding symbol.
 
\subsection{The GIT quotient of the morphism $A$}

Identify $V \cong V^{\vee} \cong \CC^n$  and consider the standard (resp.~dual) action of $\SL_n$ on $\oh(\CC^n)_d$ (resp.~$\oh(\CC^n)_{n(d-2)}$). Then, by Lemma \ref{L:equiv1}, the morphism $A \co \PP(\oh(\CC^n)_d)_{\Delta} \to \PP(\oh(\CC^n)_{n(d-2)})$ taking a form to its associated form is $\SL_n$-equivariant.  It is natural to ask the following question:

\begin{ques} \label{ques1}
Is it true that the image $A(f) \in \PP(\oh(\CC^n)_{n(d-2)})$ is semistable for every $f  \in \PP(\oh(\CC^n)_d)_{\Delta}$ and, moreover, that the induced morphism of GIT quotients
$$ \label{E:Abar}
\bar{A} \co \PP(\oh(\CC^n)_d)_{\Delta} \gitq \SL_n \to \PP(\oh(\CC^n)_{n(d-2)})^{\ss} \gitq \SL_n
$$
is an immersion?
\end{ques}

A positive answer to Question \ref{ques1} implies Conjecture \ref{C:main}.  Theorem \ref{T:binary} shows that this question has a positive answer in the case of binary forms. We note that the morphism $A \co \PP(\oh(\CC^n)_d)_{\Delta} \to \PP(\oh(\CC^n)_{n(d-2)})$ is not injective.  Indeed, Example \ref{E:example1} implies that for any nonzero  $a_1, \ldots, a_n  \in \CC$ the image of the form $a_1x_1^d + \cdots + a_n x_n^d$ under $A$ is $(x_1 \cdots x_n)^{d-2}$.  
In \cite{ai}, it was shown that the morphism $\bar{A}$ is generically injective for any $n \ge 2$.

\begin{theorem} \label{T:ai} \cite[Thm.~4.1]{ai}
The morphism $A$ descends to a rational map $\bar{A} \co \PP(\oh(\CC^n)_d)_{\Delta} \gitq \SL_n \dashrightarrow \PP(\oh(\CC^n)_{n(d-2)})^{\ss} \gitq \SL_n$ that is birational onto its image. 
\end{theorem}

In \cite{ai}, the above theorem was proven as an easy consequence of the following fact \cite[Prop.~4.3]{ai}:  for a generic form $f  \in \PP(\oh(\CC^n)_d)_{\Delta}$, the image $A(f)$ is nondegenerate {\rm (}i.e., $\Delta(A(f)) \neq 0${\rm )} and in particular is stable. 

\subsection{The GIT quotients of the morphisms $\nabla$ and $\bA$}
In order to address Questions \ref{ques1}, it is natural to utilize $\SL_n$-equivariant factorization \eqref{E:factorization-proj}. First, we consider the semistable locus of the Grassmannian $\Gr(n, \oh(\CC^n)_{d-1})$ with respect to the Pl\"ucker embedding
$$\begin{aligned}
\Gr(n, \oh(\CC^n)_{d-1})  & \to \PP\Bigl( \bigwedge\nolimits^n \oh(\CC^n)_{d-1}\Bigr), \\
 [W \subseteq \oh(\CC^n)_{d-1}]  &\mapsto \left[\bigwedge\nolimits^n W \subseteq \bigwedge\nolimits^n \oh(\CC^n)_{d-1}\right].
 \end{aligned}$$
In fact, since every very ample line bundle on the Grassmannian is a positive power of the line bundle obtained via the Pl\"ucker embedding and every such line bundle $\cL$ has a unique $\SL_n$-linearization (i.e., a choice of an $\SL_n$-action on the space of global sections $U:=\Gamma(\Gr(n, \oh(\CC^n)_{d-1}), \cL)$ such that $\Gr(n, \oh(\CC^n)_{d-1}) \subseteq \PP(U^{\vee})$ is $\SL_n$-equivariant), it follows that there are well-defined loci $\Gr(n, \oh(\CC^n)_{d-1})^{\ss}$, $\Gr(n, \oh(\CC^n)_{d-1})^{\ps}$, and $\Gr(n, \oh(\CC^n)_{d-1})^{\s}$ independent of the choice of an equivariant embedding.

Next, observe that the rational map $\nabla \co \PP(\oh(\CC^n)_d) \dashrightarrow \Gr(n,\oh(\CC^n)_{d-1})$ is defined on semistable forms. Indeed, suppose that $f \in \oh(\CC^n)_d$ has linearly dependent partial derivatives $f_{x_{{}_1}},\dots,f_{x_{{}_n}}$. Upon passing to a linearly equivalent form, we can assume that $f_{x_{{}_n}} = 0$. Consider the 
1-parameter subgroup  $\rho \co \CC^* \to \SL_n, t \mapsto \hbox{diag}(t^{-1}, \ldots, t^{-1}, t^{n-1})$.  Letting $t\ra 0$, we see that the origin lies in the closure of the $\rho$-orbit of $f$, which implies that the image of $f$ in $\PP(\oh(\CC^n)_d)$ is not semistable.

Therefore, in the spirit of Question \ref{ques1}, we can ask:
\begin{ques} \label{ques2}
Is it true that the image $\nabla(f) \in \Gr(n,\oh(\CC^n)_{d-1})$ is semistable for every semistable form $f  \in \PP(\oh(\CC^n)_d)$ and that the induced morphism of GIT quotients
$$\bar{\nabla}\co\PP(\oh(\CC^n)_d)^{\ss} \gitq \SL_n \to \Gr(n, \oh(\CC^n)_{n(d-2)})^{\ss} \gitq \SL_n$$
is a closed immersion?
\end{ques}

In the case of binary forms, Propositions \ref{P:gradient-ss} and \ref{C:closed}
show that this question has a positive answer. Recently, M. Fedorchuk has proven that $\nabla$ preserves semistability for any $n$ (see \cite{F}).

We note that in order to show that Question \ref{ques1} has a positive answer (and thus Conjecture \ref{C:main} holds), it suffices to prove that there is a factorization
$$
\bar{A}\co \PP(\oh(\CC^n)_d)_{\Delta} \gitq \SL_n \xrightarrow{\bar{\nabla}} \Gr(n, \oh(\CC^n)_{d-1})_{\Res} \gitq \SL_n\xrightarrow{\bar{\bA}} \PP(\cO(\CC^n)_{n(d-2)})^{\ss} \gitq \SL_n
$$
where both morphisms are immersions. In the above factorization, $\bar{\nabla}$ is always well-defined, and a positive answer to Question \ref{ques2} would imply that it is a closed immersion.  On the other hand, Proposition \ref{P:A-immersion} yields that $\bA\co \Gr(n, \oh(\CC^n)_{d-1})_{\Res}\rightarrow \PP(\cO(\CC^n)_{n(d-2)})$ is an immersion.  To check that there is an induced morphism $\bar{\bA}$, one would need to verify that for any $W \in \Gr(n, \oh(\CC^n)_{d-1})_{\Res}$ the image $\bA(W)$ is semistable. Moreover, to show that $\bar{\bA}$ is an immersion, it suffices to check that for every polystable subspace $W \in \Gr(n, \oh(\CC^n)_{d-1})_{\Res}$  the image $\bA(W)$ is polystable. We take precisely this approach to prove Theorem \ref{T:binary} for binary forms:  $\bar{\nabla}$ is shown to be a closed immersion in Proposition \ref{C:closed} and $\bA$ is shown to map $\Gr(2, \oh(\CC^2)_{d-1})_{\Res}$ to the semistable locus and to preserve polystability in Propositions \ref{nonzerocatal}, \ref{P:A-iso}, respectively. We note that in paper \cite{F} mentioned above the fact that for $W \in \Gr(n, \oh(\CC^n)_{d-1})_{\Res}$ the image $\bA(W)$ is semistable was obtained for any $n$.

\section{The case of binary forms}\label{casebinforms}
\setcounter{equation}{0}

From this section onwards we assume that $n=2$. In this situation, the variables are denoted by $x,y$. As always, we assume that $d\ge 3$. 

\subsection{The catalecticant}\label{mainressect} \label{S:catalecticant}
Let
$
f=\sum_{i=0}^{2N}{2N \choose i}a_ix^{2N-i}y^i\in{\mathcal O}(\CC^2)_{2N}
$
be a binary form of even degree $2N$. The catalecticant of $f$ is then defined as follows:
$$
\Cat(f):=\det\left(\begin{array}{cccc}
a_0 & a_1 & \dots & a_N\\
a_1 & a_2 & \dots & a_{N+1}\\
\vdots & \vdots & \ddots & \vdots\\
a_N & a_{N+1} & \dots & a_{2N}
\end{array}
\right).
$$
It is $\SL_2$-invariant and does not vanish if and only if the $N+1$ partial derivatives of $f$ of order $N$ are linearly independent in $\oh(\CC^2)_N$. Notice that the catalecticant is defined on the target space of the morphism $\bA \co (\oh(\CC^2)_{d-1}^{\oplus 2})_{\Res} \to \oh(\CC^2)_{2(d-2)}$.  

\begin{proposition}\label{nonzerocatal} \label{P:cat}
For any\, ${\mathbf f}\in (\oh(\CC^2)_{d-1}^{\oplus 2})_{\Res}$, we have $\Cat({\mathbf A}({\mathbf f}))\ne 0$.
\end{proposition}

\begin{proof} 
Write $\bbf = (f_1, f_2)$.  Proposition \ref{P:inverse-system} implies that there is an isomorphism $M(f) = \CC[x, y]/ (f_1, f_2) \cong \CC[x,y] / \bA(\bbf)^{\perp}$. Since $\dim M(f)_{d-2} = d-1$, we have $\dim (\CC[x,y] / \bA(\bbf)^{\perp})_{d-2} = d-1$ or in other words $(\bA(\bbf)^{\perp})_{d-2} = 0$.  It follows that the $d-1$ partial derivatives of $\bA(\bbf)$ of order $d-2$ are linearly independent.
\end{proof}

\begin{remark} The statement that for $f \in \oh(\CC^2)_{d, \Delta}$ one has $\Cat(A(f)) \neq 0$ had been established in \cite[Prop.~5.1]{ai}.
\end{remark}

Proposition \ref{nonzerocatal} shows that $\bA$ is a morphism of affine varieties
\begin{equation} \label{E:A-cat}
\bA \co \Gr(2, \oh(\CC^2)_{d-1})_{\Res} \to \PP(\oh(\CC^2)_{2(d-2)})_{\Cat},
\end{equation}
with $\PP(\oh(\CC^2)_{2(d-2)})_{\Cat}$ being the affine open subvariety of $\PP(\oh(\CC^2)_{2(d-2)})$ where the catalecticant does not vanish.

\subsection{The morphism $\bA$ is an isomorphism}

\begin{prop} \label{P:A-iso} \it
The morphism $\bA \co \Gr(2, \oh(\CC^2)_{d-1})_{\Res} \to \PP(\oh(\CC^2)_{2(d-2)})_{\Cat}$ is an isomorphism.  
\end{prop}

\begin{proof}
Recall from Section \ref{sectimage} that $U_{\Res} \subseteq \PP(\oh(\CC^2)_{2(d-2)})$ denotes the locus of forms $F$ such that $\dim F^{\perp} \cap \oh(\CC^2)_{d-1} = 2$ and the subspace $F^{\perp} \cap \oh(\CC^2)_{d-1}$ has nonzero resultant.  
By Proposition \ref{P:A-immersion}, the morphism $\bA \co  \Gr(2, \oh(\CC^2)_{d-1})_{\Res} \to U_{\Res}$ is an isomorphism, and by Proposition \ref{P:cat}, we have the inclusion $U_{\Res} \subseteq \PP(\oh(\CC^2)_{2(d-2)})_{\Cat}$. To obtain the other inclusion, fix a form $F \in \oh(\CC^2)_{2(d-2)}$ with $\Cat(F) \neq 0$ and let $W:=F^{\perp} \cap \oh(\CC^2)_{d-1}$. First, observe that $\dim F^{\perp} \cap \oh(\CC^2)_{d-2} = 0$, which necessarily implies $\dim F^{\perp} \cap \oh(\CC^2)_{i} = 0$ for all $i \le d-2$.  Therefore, for $j=0, \ldots, d-2$, we have $\dim (\CC[x,y]/F^{\perp})_{d-2+j} = \dim (\CC[x,y]/F^{\perp})_{d-2-j} =d-1-j$, where we used the symmetry of the Hilbert function, and thus  $\dim F^{\perp} \cap \oh(\CC^2)_{d-2+j}= 2j$.   Taking $j=1$, we obtain that $W$ is 2-dimensional.  

Next, if $W=\langle f_1, f_2 \rangle$, we need to prove that $\Res(f_1,f_2)\ne 0$. If this is not the case, then  $f_1$, $f_2$ have a common factor $p$, and we can write $f_1 = p h_1$,\linebreak $f_2 = p h_2$ where $h_1$, $h_2$ are forms of some positive degree $m$ without common linear factors. Since $h_1$, $h_2$ have nonzero resultant, it follows from Remark \ref{R:hilbert-function} that $\dim (h_1, h_2) \cap \oh(\CC^2)_{m-1+j} = 2j$ for $j=0, \ldots, m-1$.  We then obtain $\dim (f_1, f_2) \cap \oh(\CC^2)_{d-2+j} = 2j$ for $j=0, \ldots, m-1$. In particular, the inclusion of ideals $(f_1, f_2) \subseteq F^{\perp}$ is an equality in degree $d+m-3$. Now choose linear factors $l_1 \, | \, h_1$, $l_2 \, | \, h_2$ and set $q : = p \frac{h_1}{l_1} \frac{h_2}{l_2}$, which is a form of degree $d+m-3$. Since $l_1 q, l_2 q \in F^{\perp}$, it follows that $q\in F^{\perp}$. Hence $m>1$ since otherwise $q$ is in $F^{\perp}\cap{\mathcal O}(\CC^2)_{d-2}=0$. As $q$ lies in the ideal $(f_1, f_2)$, writing $q = a_1 f_1 + a_2 f_2$ for $a_1, a_2 \in\CC[x,y]_{m-2}$ and factoring out $p$ yields $\frac{h_1}{l_1} \frac{h_2} {l_2} = a_1 h_1 + a_2 h_2$, which contradicts the fact that $h_1/l_1$ and $h_2/l_2$ have no common linear factors. \end{proof}

Taking the GIT quotients of the factorization provided by \eqref{E:factorization-proj} and \eqref{E:A-cat}, we obtain the factorization
$$
\bar{A}: \PP(\oh(\CC^2)_d)_{\Delta} \gitq \SL_2 \xrightarrow{\bar{\nabla}} \Gr(2, \oh(\CC^2)_{d-1})_{\Res} \gitq \SL_2\xrightarrow{\bar{\bA}} \PP(\cO(\CC^2)_{2(d-2)})_{\Cat} \gitq \SL_2.
$$
By Proposition \ref{P:A-iso}, the morphism $\bar{A}$ is a closed immersion if and only if $\bar{\nabla}$ is a closed immersion.  We conclude:

\begin{cor} \label{C:equiv-conj} \it
Theorem {\rm \ref{T:binary}} is equivalent to the morphism 
$$\bar{\nabla} \co \PP(\oh(\CC^2)_d)_{\Delta} \gitq \SL_2 \to \Gr(2, \oh(\CC^2)_{d-1})_{\Res} \gitq \SL_2$$ being a closed immersion.
\end{cor}

In Section \ref{propgamm}, we will establish Theorem {\rm \ref{T:binary}} by showing that in fact the induced quotient morphism $\bar{\nabla} \co \PP(\oh(\CC^2)_d)^{\ss}\gitq \SL_2 \to \Gr(2, \oh(\CC^2)_{d-1})^{\ss} \gitq \SL_2$ of the entire semistable loci is a closed immersion.

\section{The morphism $\nabla$ and stability }\label{S:stability}
\setcounter{equation}{0}

The rational map $\nabla \co \PP(\oh(\CC^2)_d) \dashrightarrow \Gr(2, \oh(\CC^2)_{d-1})$ given by $f \mapsto \langle f_x, f_y \rangle$ is well-defined  on forms that are not powers of linear forms and, in particular, on semistable forms.   In this section, we show that if $f$ is semistable, then so is $\nabla(f)$ (Proposition \ref{P:gradient-ss}) and that if $f$ is polystable, then so is $\nabla(f)$ (Proposition \ref{P:gradient-ps}).  On the other hand, it is not true that if $f$ is stable, then $\nabla(f)$ is stable; indeed $\nabla(x^d+y^d)$ is the subspace $\langle x^{d-1}, y^{d-1} \rangle$, which is fixed by $\CC^*$ and thus is not stable.

These results will allow us to show that the induced morphism of quotients 
$$\bar{\nabla} \co \PP(\oh(\CC^2)_d)^{\ss}\gitq \SL_2 \to \Gr(2, \oh(\CC^2)_{d-1})^{\ss} \gitq \SL_2$$
is finite and injective (Corollary \ref{C:finite-injective}).  In Section \ref{propgamm}, we will prove that its image is normal, which will lead to the conclusion that $\bar{\nabla}$ is a closed immersion. 

On the other hand, we feel that investigating the stability of subspaces in the Grassmannian and the preservation of stability by the morphism $\nabla$ are independently interesting. In particular, we include a geometric characterization of the semi(stable) locus in $\Gr(2, \oh(\CC^2)_{d-1})$, which is used in the proof of Theorem \ref{T:binary} below (but is not strictly necessary for it). 
  
\subsection{Semistability of subspaces}\label{semistabilsubspaces} We first study the (semi)stabi\-li\-ty of elements of the Grassmannian $\Gr(2, \oh(\CC^2)_m)$ for any $m\ge 1$ by utilizing the Hilbert-Mumford criterion (cf.~\cite[Thm.~2.1]{git}).  By this criterion, a subspace $W \in\Gr(2, \oh(\CC^2)_m)$ is (semi)stable if and only if it is (semi)stable with respect to the action of every one-parameter subgroup $\rho \co {\mathbb C}^*  \to \SL_2$. The $\rho$-(semi)stability of $W$ is understood by calculating the Hilbert-Mumford index $\mu(W,\rho)$. It can be computed as follows.   

Choose any coordinates $x,y$ in $\CC^2$ in which $\rho$ is given by diagonal matrices
$$
\rho \co {\mathbb C}^*  \to \SL_2,\quad t\mapsto\begin{pmatrix} t^{\tau} & 0 \\ 
0 &  t^{-\tau}
   \end{pmatrix},
$$
with $\tau\in\ZZ$, and for every monomial $x^{m-i}y^i\in \oh(\CC^2)_m$, define its weight as $\wt(x^{m-i}y^i) := \tau(m-i)-\tau i=\tau(m-2i)$. Consider now a 2-dimensional subspace $W =\langle f_1,f_2\rangle$ and write the basis forms $f_1,f_2$ in the coordinates $x, y$ as
\begin{equation}
\begin{aligned}
f_1  = \sum_{i=0}^m {m \choose i} a_i x^{m-i}y^i,\quad f_2  = \sum_{i=0}^m  {m \choose i}b_i x^{m-i}y^i.\label{pandq}
\end{aligned}
\end{equation}
Under the Pl\"ucker embedding of $\Gr(2,\oh(\CC^2)_m)$ in $\PP(\bigwedge\nolimits^2  \oh(\CC^2)_m)$ the subspace $W$ is mapped to the line spanned by
$$
f_1\wedge f_2=\sum_{0\le i<j\le m}{m \choose i}{m \choose j}(a_ib_j-a_jb_i)x^{m-i}y^i\wedge x^{m-j}y^j.\label{f1wedgef2}
$$
Then the Hilbert-Mumford index of $W$ is computed as
\begin{equation}
\makebox[250pt]{$\begin{array}{l}
\hspace{0.8cm}\mu(W,\rho)=\max_{0\le i<j\le m}\left\{-(\wt(x^{m-i}y^i)+\wt(x^{m-j}y^j))\mid (a_ib_j-a_jb_i)\ne 0\right\},
\end{array}$}\label{hilmumindcom}
\end{equation}
and the subspace $W$ is $\rho$-semistable (resp.~$\rho$-stable) if and only if $\mu(W,\rho)\,\ge0$ (resp.~$\mu(W,\rho)\,>0$). 

Thus, $W$ is semistable (resp.~stable) if and only
if for every choice of coordinates $x,y$ in $\CC^2$ and $\tau\in\ZZ$ there exist two monomials $x^{m-i}y^i$, $x^{m-j}y^j$ in $\oh(\CC^2)_m$ such that the sum of their weights is nonpositive (resp.~ne\-gative) and the vectors $(a_i,b_i)$, $(a_j,b_j)$ of the corresponding coefficients of a basis of $W$ written in the coordinates $x,y$ are linearly independent. It is clear that it suffices to consider only one-parameter subgroups $\rho \co \CC^* \to \SL_2$ that can be diagonalized with $\tau=-1$, in which case formula (\ref{hilmumindcom}) turns into
\begin{equation}
\mu(W,\rho)=\max_{0\le i<j\le m}\left\{2(m-i-j)\mid(a_ib_j-a_jb_i)\ne 0\right\}.\label{hilmumindcom1}
\end{equation}

\subsection{A geometric characterization of the semistability of subspaces} \label{S:geometric-characterization}
It is well-known that $f \in \PP(\oh(\CC^2)_m)$ is semistable (resp.~stable) if and only if $f$ has no roots of multiplicity $>m/2$ (resp.~$\ge m/2$). If $f$ is strictly semistable (i.e., semistable but not stable), then there is a root of multiplicity $m/2$ (in particular $m$ is even), which implies that the unique closed $\SL_2$-orbit in $\PP(\oh(\CC^2)_m)^{\ss}$ that lies in the closure of the orbit of $f$ is the orbit of $x^{m/2}y^{m/2}$. We will now give an analogous characterization of the (semi)stability of a subspace $W\in\Gr(2,\oh(\CC^2)_m)$ using the Hilbert-Mumford criterion as detailed in Section \ref{semistabilsubspaces}.
 
\begin{proposition} \label{P:geometric-characterization}
Let $W\in\Gr(2,\oh(\CC^2)_m)$.  Then $W$ is semistable {\rm (}resp.~stable{\rm )} if and only if there do not exist integers $0 \le i < j$ with $i +j>m$ {\rm (}resp.~$i+j\ge m${\rm )} and a linear form $L$ such that $L^i$ divides every $p\in W$ and there is $q \in W$ divisible by $L^j$.

In particular, a semistable subspace $W$ is strictly semistable if and only if there exist a nonnegative integer $i < m/2$ and a linear form $L$ such 
that $L^{i}$ divides every $p\in W$ and there is $q \in W$ divisible by $L^{m-i}$.  In this case, the unique closed $\SL_2$-orbit in $\Gr(2,\oh(\CC^2)_m)^{\ss}$ that lies in the closure of the orbit of $W$ is the orbit of the subspace $\langle x^{i}y^{m-i}, x^{m-i}y^{i} \rangle$.
\end{proposition}

\begin{proof}
Let $\rho \co \CC^* \to \SL_2$ be a one-parameter subgroup diagonalizable with $\tau=-1$ and $x,y$ coordinates in $\CC^2$ in which $\rho$ is diagonal. Fix an element $W\in\Gr(2,\oh(\CC^2)_m)$. Choose basis forms $f_1, f_2$ in $W$, writing them as in (\ref{pandq}), and consider the matrix
\begin{equation} \label{eqn-matrix}
\begin{pmatrix} 	a_0 & a_1 & a_2 & \cdots & a_{m} \\
 				b_0 & b_1 & b_2 & \cdots & b_{m}
\end{pmatrix},
\end{equation}
where the columns are enumerated from $0$ to $m$. Let $k$ be the smallest integer such that the $k$th column is nonzero and $\ell$ be the smallest integer such that the $\ell$th column is linearly independent with the $k$th column. Then by formula (\ref{hilmumindcom1}) we have
\begin{equation}
\mu(W, \rho)=2(m-k-\ell).\label{formulaformu}
\end{equation}
We thus conclude that $W$ is $\rho$-semistable (resp.~$\rho$-stable) if and only if
$k+\ell \le m$ (resp.~$k+\ell< m$). Then the characterization of semistability (resp.~stability) follows immediately. Here the necessity implication is shown by choosing coordinates such that $L=y$ and using the subgroup
\begin{equation}
 {\mathbb C}^*  \to \SL_2,\quad
   t  \mapsto 
   \begin{pmatrix} t^{-1} & 0  \\ 
0 & t
   \end{pmatrix}\label{1PS}
\end{equation}
to destabilize $W$.

To prove the statement contained in the last sentence of the proposition, fix a strictly semistable subspace $W$, with $0\le i< m/2$ and $L$ being the corresponding integer and linear form, and let $x,y$ be the standard coordinates on $\CC^2$. By passing to another subspace in the orbit of $W$, we can assume that $L=y$. We may choose a basis $f_1,f_2$, written as in (\ref{pandq}), so that matrix (\ref{eqn-matrix}) is of the form
$$
\begin{pmatrix}
	0 & \cdots & 0 & a_i & a_{i+1} & \cdots & a_{m-i-1} & a_{m-i}  & \cdots & a_m\\
 	0 & \cdots & 0 & 0 & 0 & \cdots & 0 & b_{m-i} & \cdots & b_{m}
\end{pmatrix}.
$$
Let $k$ and $\ell$ be the integers defined as above, and $\rho$ the one-parameter subgroup introduced in (\ref{1PS}). If $a_{i} = 0$, then $k > i$. As $\ell \ge m-i$, we have $k+\ell > m$ contradicting the $\rho$-semistability of $W$. 
Similarly if $b_{m-i} = 0$, then $\ell > m-i$ so $k+\ell > m$, which again contradicts the $\rho$-semistability of $W$. 
We conclude that both $a_i$ and $b_{m-i}$ are nonzero.
An easy calculation now shows
$$
\hspace{-0.1cm}\rho(t)^{-1}\, W=\big\langle y^i \big(a_i  x^{m-i}+\cdots+t^{2(m-i)}a_{m}y^{m-i}\big),\hspace{-0.05cm}  y^{m-i}( b_{m-i}  x^i+\dots+t^{2i}b_m y^i) \big\rangle. 
$$
As $a_i$ and $b_{m-i}$ are nonzero, $\lim_{t \to 0} \rho(t)^{-1}\,W = \langle x^iy^{m-i},x^{m-i}y^i \rangle$. Seeing that the orbit of $\langle x^iy^{m-i},x^{m-i}y^i \rangle$ is closed in $\Gr(2,\oh(\CC^2)_m)^{\ss}$ is straightforward. \end{proof}

\subsection{Preservation of semistability by $\nabla$}

\begin{proposition} \label{P:gradient-ss}  If a form $f \in \PP(\oh(\CC^2)_d)$ is semistable, then so is the subspace $\nabla(f) \in \Gr(2, \oh(\CC^2)_{d-1})$.
\end{proposition}

\begin{proof} Let $\rho \co \CC^* \to \SL_2$ be a one-parameter subgroup and $x,y$ coordinates in $\CC^2$ in which $\rho$ is diagonal with $\tau=-1$.  
We claim that the following stronger statement holds: the $\rho$-semistability of a form $f \in \PP(\oh(\CC^2)_d)$ not linearly equivalent to $x^d$ implies the $\rho$-semistability of $\nabla(f)$.
 Write $f = \sum_{i=0}^d {d \choose i} a_i x^{d-i} y^{i}$
and consider the matrix
\begin{equation}
\begin{pmatrix} 	a_0 & a_1 & a_2 & \cdots & a_{d-1} \\
 				a_1 & a_2 & a_3 & \cdots & a_{d}
\end{pmatrix},\label{matrixn=2}
\end{equation}
with the columns enumerated from $0$ to $d-1$. The columns of this matrix are proportional to those of (\ref{eqn-matrix}) with $m=d-1$, $f_1=f_x$, $f_2=f_y$. Let $i$ be the smallest integer with $a_i \neq 0$.  If $i > 0$, then columns $i$ and $i-1$ are linearly independent, and by formula (\ref{formulaformu}) the subspace $\nabla(f)$ is $\rho$-semistable if and only if $2i \le d$. This last condition is implied by (and in fact equivalent to) the $\rho$-semistability of $f$.
If $i=0$,  then, since matrix (\ref{matrixn=2}) has full rank, the 0th column is linearly independent with the $j$th column for some $j=1, \ldots, d-1$, hence $\nabla(f)$ is $\rho$-semistable.
\end{proof}

\begin{remark} Observe that the converse of the proposition is also true, although we will not use this fact.  Namely, the above proof shows that if $f \in \PP(\oh(\CC^2)_d)$ is not the $d$th power of a linear form, then $f$ is $\rho$-semistable if and only if $\nabla(f)$ is. Note that the converse to Proposition \ref{P:gradient-ss} also follows easily from the fact that the $\SL_2$-linearization of $\oh(k)$ on $\PP(\oh(\CC^2)_d)$ for some $k>0$ is the pullback of the $\SL_2$-linearization of the line bundle corresponding to the Pl\"ucker embedding via the rational map $\nabla\co\PP(\oh(\CC^2)_d) \dashrightarrow \Gr(2, \oh(\CC^2)_{d-1})$.
\end{remark}

Proposition \ref{P:gradient-ss} implies that the morphism $\nabla\co\PP(\oh(\CC^2)_d)_{\Delta} \to \Gr(2, \oh(\CC^2)_{d-1})_{\Res}$ extends to a morphism (which we denote by the same symbol) $\nabla\co\PP(\oh(\CC^2)_d)^{\ss} \to \Gr(2, \oh(\CC^2)_{d-1})^{\ss}$, so the morphism $\bar{\nabla}\co\PP(\oh(\CC^2)_d)_{\Delta} \gitq \SL_2 \to \Gr(2, \oh(\CC^2)_{d-1})_{\Res}$ of affine GIT quotients extends to the induced morphism $\bar{\nabla} \co \PP(\oh(\CC^2)_d)^{\ss} \gitq \SL_2 \to \Gr(2, \oh(\CC^2)_{d-1})^{\ss} \gitq \SL_2$ of projective GIT quotients.

\subsection{Preservation of polystability by $\nabla$} 

\begin{proposition}  \label{P:gradient-ps}
If a form $f \in \PP(\oh(\CC^2)_d)$ is polystable, then so is the subspace $\nabla(f) \in \Gr(2, \oh(\CC^2)_{d-1})$.
\end{proposition}

\begin{proof} 
By Proposition \ref{P:gradient-ss}, we know that $\nabla(f)$ is semistable.  Assume that $\nabla(f)$ is strictly semistable. Then by the Hilbert-Mumford criterion there exists a one-parameter subgroup $\rho \co\CC^* \to \SL_2$ with $\mu(\nabla(f), \rho) = 0$ and a choice of coordinates $x, y$ with $\tau=-1$. Write $f = \sum_{i=0}^d {d \choose i} a_i x^{d-i} y^{i}$ and consider matrix (\ref{matrixn=2}). Let $k$ be the smallest integer such that the $k$th column of the matrix is nonzero and $\ell$ the smallest integer such that the $\ell$th column is linearly independent with the $k$th column. Then by formula (\ref{formulaformu}) with $m=d-1$ we have $k+\ell = d-1$.  

If $k=0$, then $\ell=d-1$, and by a short computation one verifies that $f$ is linearly equivalent to $x^d + y^d$. In this case, $\nabla(f)$ lies in the orbit of $\langle x^{d-1}, y^{d-1} \rangle$, which is polystable. If $k \neq 0$, then $k<d-1$ and the $k$th column is $\begin{pmatrix} 0 \\ a_{k+1} \end{pmatrix}$ with $a_{k+1} \neq 0$. Then the $(k+1)$th column $\begin{pmatrix} a_{k+1} \\ a_{k+2} \end{pmatrix}$ is linearly independent with the $k$th column.  Hence, $\ell=k+1$, which implies that $d$ is even and $k = d/2-1$. It then follows that $y^{d/2}$ divides $f$, i.e., $f$ is strictly semistable. Since $f$ is polystable, it is then linearly equivalent to $x^{d/2}y^{d/2}$. Therefore, $\nabla(f)$ lies in the orbit of the subspace $\langle x^{d/2-1}y^{d/2}, x^{d/2} y^{d/2-1} \rangle$, which is again polystable.\end{proof}

\begin{cor} \label{C:finite-injective} \it  The morphism
$\bar{\nabla} \co \PP(\oh(\CC^2)_d)^{\ss} \gitq \SL_2 \to \Gr(2, \oh(\CC^2)_{d-1})^{\ss} \gitq \SL_2$ is finite and injective.
\end{cor}

\begin{proof}
Suppose that $\bar{f}, \bar{g} \in \PP(\oh(\CC^2)_d)^{\ss} \gitq \SL_2$ are such that $\bar{\nabla}(\bar{f}) = \bar{\nabla}(\bar{g})$.  Let $f, g \in \PP(\oh(\CC^2)_d)^{\ss}$ be any polystable preimages of $\bar{f}, \bar{g}$.  By the proposition, the subspaces $\nabla(f), \nabla(g) \in \Gr(2, \oh(\CC^2)_{d-1})$ are polystable hence linearly equivalent (as they map to same point in $\Gr(2, \oh(\CC^2)_{d-1})^{\ss} \gitq \SL_2$). Then, by \cite[Prop.~1.1]{D} the forms $f, g$ are linearly equivalent. Thus 
$\bar{f} = \bar{g}$, and we have established that $\bar{\nabla}$ is injective.  Since $\bar{\nabla}$ is an injective morphism of projective varieties, is it also finite.
\end{proof}

To prove Theorem \ref{T:binary}, it therefore suffices to show that the (closed) image of $\bar{\nabla}$ is normal. This is accomplished in the next section.

\section{Normality of the image of $\bar{\nabla}$ and the proof of Theorem \ref{T:binary}}\label{propgamm}
\setcounter{equation}{0}

In this section, we finalize the proof of Theorem \ref{T:binary} by showing that the image $\bar{\nabla} (\PP(\oh(\CC^2)_d)^{\ss} \gitq \SL_2)$ of $\bar{\nabla}$ is normal  (Corollary \ref{C:normal}).

\subsection{The injectivity of the differential of $\nabla$ and the locus where $\nabla$ is a closed immersion} \label{S:differential} As the morphism $\nabla \co \PP(\oh(\CC^2)_d) \setminus \SL_2 x^d \to \Gr(2, \oh(\CC^2)_{d-1})$
is induced by the linear map $\grad\co \oh(\CC^2)_d \to \oh(\CC^2)_{d-1}^{\oplus 2}$, it is easy to compute its differential. Indeed, any element $f \in \PP(\oh(\CC^2)_d)$ can be regarded as a line in $\oh(\CC^2)_d$, and there is a natural identification 
$$
T_f \PP(\oh(\CC^2)_d) = \Hom (f, \oh(\CC^2)_d /f) = \oh(\CC^2)_d / f.
$$
Likewise, for $W\in\Gr(2, \oh(\CC^2)_{d-1})$, there is a natural identification
$$
T_{W} \Gr(2, \oh(\CC^2)_{d-1}) = \Hom
(W, \oh(\CC^2)_{d-1}/ W).
$$
Let $f \in \PP(\oh(\CC^2)_d) \setminus \SL_2  x^d$.
The differential $D_f \nabla$ of $\nabla$  is described, using the above identifications, as
$$\begin{aligned}
D_f\nabla \co \oh(\CC^2)_d / f  & \to  \Hom
\big(\langle f_x,f_y\rangle, \oh(\CC^2)_{d-1}/\langle f_x,f_y\rangle \big) \\
g & \mapsto \chi,\, \hbox{with $\chi(\alpha f_x + \beta f_y) = \alpha g_x + \beta g_y$},
\end{aligned}
$$
where $g_x, g_y$ are interpreted as the images in $\oh(\CC^2)_{d-1}/\langle f_x,f_y\rangle$ of the first-order partial derivatives of a lift of $g$ to $\oh(\CC^2)_d$. We conclude that the differential $D_f\nabla$ is injective if and only if there does not exist $g \in\PP(\oh(\CC^2)_d)$, $g\ne f$, such that $\langle g_x,g_y\rangle \subseteq \langle f_x,f_y\rangle$.

We now need the following elementary lemma, which we state without proof.

\begin{lemma} \label{P:binary-linearly-equivalent} \it Let $f \in \oh(\CC^2)_d$ be a form not linearly equivalent to $x^{d-1}y$. If $\langle f_x,f_y\rangle \supseteq\langle g_x,g_y\rangle$ for some $g\in \oh(\CC^2)_d$ not proportional to $f$, then $f$ is linearly equivalent to $x^d + y^d$.
\end{lemma} 

\noindent Using Lemma \ref{P:binary-linearly-equivalent} we obtain:

\begin{proposition} \label{P:unramified}
If $f \in \PP(\oh(\CC^2)_d)^{\ss}$ is not in\, $\SL_2(x^d + y^d)$, then $D_f\nabla$ is injective.
\end{proposition}

Next, consider the commutative diagram:
$$
\xymatrix{
\PP(\oh(\CC^2)_d)^{\ss} \ar[r]^{\hspace{-0.8cm}{\nabla}} \ar[d]_{\pi_1} 		& \Gr(2, \oh(\CC^2)_{d-1})^{\ss} \ar[d]^{\pi_2} \\
\PP(\oh(\CC^2)_d)^{\ss} \gitq \SL_2	\ar[r]^{\hspace{-0.8cm}{\bar{\nabla}}}			& \Gr(2, \oh(\CC^2)_{d-1})^{\ss} \gitq \SL_2.
}
$$
The morphism $\nabla$ is not a closed immersion as it is not even injective---all the forms $ax^d + b y^d$ have the same image under $\nabla$ for $a,b \neq 0$.  The following proposition says that it is in fact a closed immersion once we remove the orbit $\SL_2  (x^d+y^d)$ in the source and the  locus $\pi_2^{-1}(\pi_2( \langle x^{d-1}, y^{d-1}\rangle))$ in the target. Note that we have $\nabla^{-1}\big(\pi_2^{-1}(\pi_2( \langle x^{d-1}, y^{d-1}\rangle))\big)=\SL_2  (x^d+y^d)$ by Lemma \ref{P:binary-linearly-equivalent} and Remark \ref{R:codim1} below.

\begin{proposition}\label{P:generically-closed}
The morphism 
\begin{equation} \label{E:restriction}
\PP(\oh(\CC^2)_d)^{\ss} \setminus \SL_2  (x^d+y^d) \to \Gr(2, \oh(\CC^2)_{d-1})^{\ss} \setminus \pi_2^{-1}(\pi_2( \langle x^{d-1}, y^{d-1}\rangle))
\end{equation}
obtained by restricting $\nabla$ is a closed immersion. 
\end{proposition}

\begin{proof} 
Morphism \eqref{E:restriction} induces the morphism 
\begin{equation} \label{E:quotient-restriction}
\begin{array}{l}
\big(\PP(\oh(\CC^2)_d)^{\ss} \gitq \SL_2 \big) \setminus \pi_1(x^d+y^d)	\to\\
\vspace{-0.3cm}\\
 \hspace{4cm}\big(\Gr(2, \oh(\CC^2)_{d-1})^{\ss} \gitq \SL_2 \big) \setminus \pi_2(\langle x^{d-1}, y^{d-1}\rangle)
 \end{array}
\end{equation}
of GIT quotients. Morphism \eqref{E:quotient-restriction} is finite by Corollary \ref{C:finite-injective} whereas morphism \eqref{E:restriction} is injective by Lemma \ref{P:binary-linearly-equivalent} and maps closed orbits to closed orbits by Proposition \ref{P:gradient-ps}. Therefore, morphism \eqref{E:restriction} is finite (cf.~\cite[p.~89]{luna}).  It follows that \eqref{E:restriction}  is a closed immersion since by Proposition \ref{P:unramified} it is unramified.
\end{proof}

\begin{remark} \label{R:codim1}
We note that $\pi_2^{-1} (\pi_2 ( \langle x^{d-1}, y^{d-1} \rangle))\subseteq\Gr(2, \oh(\CC^2)_{d-1})^{\ss}$ consists of all subspaces $\langle L^{d-1}, f \rangle$ where $L\ne 0$ is a linear form and $f\ne 0$ is a form not divisible by $L$ (cf.~Proposition \ref{P:geometric-characterization}). Therefore, $\pi_2^{-1}(\pi_2 ( \langle x^{d-1}, y^{d-1} \rangle))$ is the image of the rational map $\oh(\CC^2)_1 \times \oh(\CC^2)_{d-1} \dashrightarrow \Gr(2,\oh(\CC^2)_{d-1})$ defined by $(L, f) \mapsto \langle L^{d-1}, f \rangle$. Hence, this subvariety is irreducible of dimension $d-1$. 
\end{remark}

\subsection{Normality of the image of $\nabla$} \label{S:normality}
We assume that $d \ge 4$.  For $d=3$, Theorem \ref{T:binary} (as well as Corollary \ref{C:normal} and Proposition \ref{C:closed} below) is trivial. 

Recall that the gradient morphism $\grad \co \oh(\CC^2)_d \to \oh(\CC^2)_{d-1}^{\oplus 2}$ is given by $f \mapsto (f_x, f_y)$.  Consider the action of $\GL_2$ on $\oh(\CC^2)_{d-1}^{\oplus 2}$ defined by $g (f_1, f_2) = (f_1,f_2) \cdot g^{-1}$ for $g \in \GL_2$ and  $f_1,f_2 \in \oh(\CC^2)_{d-1}$ (note that this action is the restriction of that in (\ref{doubleaction}) to the subgroup $\{{\rm id}\}\times\GL_2\subseteq\GL_2\times\GL_2$). Introduce the irreducible variety
$$ \label{E:Wd}
W_d = \bar{ \GL_2  \grad(\oh(\CC^2)_d )} \subseteq \oh(\CC^2)_{d-1}^{\oplus 2}.
$$

Recall also that we have the commutative diagram
$$
\xymatrix{
\oh(\CC^2)_{d}  \setminus \overline{\GL_2 x^d} \ar[r]^{\hspace{-2.7cm}\grad} \ar[d]		&\oh(\CC^2)_{d-1}^{\oplus 2} \setminus  \{ (f_1, f_2) \mid f_1, f_2 \text{ are linearly dependent}\}  \ar[d]^{\pi} \\
\PP(\oh(\CC^2)_d) \setminus \SL_2 x^d \ar[r]^{\hspace{0.1cm}\nabla}				& \Gr(2,\oh(\CC^2)_{d-1}),
}
$$
where
$\pi$ is a $\GL_2$-quotient.  Clearly, $W_d$ coincides with the closure in $\oh(\CC^2)_{d-1}^{\oplus 2}$ of $\pi^{-1} \big(\nabla(\PP(\oh(\CC^2)_d) \setminus \SL_2  x^d)\big)$.

\begin{prop} \label{P:normal} \it
The variety $W_d$ is normal.
\end{prop}

\begin{proof}  Consider the subscheme 
$$W'_d := \{ (f_1, f_2) \, \mid \, (f_1)_{x}, (f_1)_{y}, (f_2)_{x}, (f_2)_{y} \hbox{ are linearly dependent} \} \subseteq \oh(\CC^2)_{d-1}^{\oplus 2};$$
more precisely, if we write $f_1  = \sum_{i=0}^{d-1} {d-1 \choose i} a_i x^{d-1-i}y^i$, $f_2  = \sum_{i=0}^{d-1}  {d-1 \choose i}b_i x^{d-1-i}y^i$ and give $\oh(\CC^2)_{d-1}^{\oplus 2}$ the coordinates $a_0, \ldots, a_{d-1}, b_0, \ldots, b_{d-1}$, then $W'_d$ is defined by the $4 \times 4$ minors of the matrix
$$
\begin{pmatrix}
a_0 & a_1 & a_2 & \cdots & a_{d-2} \\
a_1 & a_2 & a_3 & \cdots & a_{d-1} \\
b_0 & b_1 & b_2 & \cdots & b_{d-2} \\
b_1 & b_2 & b_3 & \cdots & b_{d-1}
\end{pmatrix}.
$$
We first claim that $W_d$ and $W'_d$ agree set-theoretically or, in other words, that
$W_d = (W'_d)_{\red}$.  Indeed, if $(f_1,f_2) = g (f_x, f_y)$ for some $g \in \GL_2$ and $f \in \oh(\CC^2)_d$, then $ (f_1)_{x}, (f_1)_{y}, (f_2)_{x}, (f_2)_{y}$ are linearly dependent.  Conversely,
if the latter condition holds, then  $(a f_1 + b f_2)_y = (c f_1 + e f_2)_x$ for some $(a, b, c, e) \neq 0$.  It follows that there is an element $f \in \oh(\CC^2)_d$ such that $(f_x, f_y) = (a f_1 + b f_2, c f_1 + e f_2)$.  Let $g = \begin{pmatrix} a & c \\ b & e \end{pmatrix}$.  If $g$ is invertible, then $(f_1, f_2) = g (f_x,f_y)$.  Otherwise, $f_x$ and $f_y$ are linearly dependent, hence $f={\ell}^d$ for some $\ell\in\oh(\CC^2)_1$. This implies that $a f_1 + b f_2$ and $c f_1 + d f_2$ are multiples of $\ell^{d-1}$. Assume that $a\ne 0$ and after changing coordinates write $a f_1+b f_2=pdx^{d-1}$, where $p$ is either 0 or 1. Consider the forms $h_t := px^d + t \int f_2 dy$ parametrized by nonzero $t \in \CC$ and set $\sigma_t :=  \begin{pmatrix} a & 0 \\ b & 1 \end{pmatrix}\begin{pmatrix} 1 & 0 \\ 0 & t \end{pmatrix}$.  Then one checks that $\lim_{t \to 0} \sigma_t( (h_t)_x, (h_t)_y)=(f_1, f_2)$, which establishes that $(f_1, f_2) \in W_d$. 

Now consider the determinantal variety $\cD \subseteq \oh(\CC^2)_{d-2}^{\oplus 4}$ parametrizing linearly dependent tuples $(f_1, f_2, f_3, f_4)$; in other words, if $\oh(\CC^2)_{d-2}^{\oplus 4}$ is given coordinates $z_{i,j}$ for $0 \le i \le 3$ and $0 \le j \le d-2$, then $\cD$ is defined by the ideal  generated by the $4 \times 4$ minors of the matrix of indeterminates
$$\begin{pmatrix} 
	z_{0,0} & z_{0,1} & \cdots & z_{0,d-2} \\
	z_{1,0} & z_{1,1} & \cdots & z_{1,d-2} \\
	z_{2,0} &	 z_{2,1} & \cdots & z_{2,d-2} \\
	z_{3,0} & z_{3,1} & \cdots & z_{3,d-2} \\
\end{pmatrix}.
$$
By \cite{eagon-hochster} and  \cite[\S 1C, \S 2B]{bruns-vetter}, $\cD$ is an irreducible Cohen-Macaulay variety of dimension $3d$.  Clearly, $W'_d$ is the intersection of $\cD$ with the subspace defined by the $2d-4$ equations: $z_{0,1}-z_{1,0}, \ldots,$ $z_{0,d-2} - z_{1, d-3}$ and $z_{2,1}-z_{3,0}, \ldots, z_{2,d-2} - z_{3, d-3}$. Further, we have $\dim W'_d = \dim W_d = d+4$. Since 
$\codim_{ \cD} W'_d = \dim \cD - \dim W'_d = 3d - (d+4) = 2d-4$
is the number of defining equations of $W'_d$ in $\cD$, we conclude that $W'_d$ is Cohen-Macaulay (cf.~\cite[Prop.~18.13]{eisenbud}).

We will now show that $W'_d$ is in fact reduced.  It suffices to prove that $W'_d$ is generically reduced since Cohen-Macaulay schemes do not have embedded components, 
and therefore it suffices to show that $W'_d$ is smooth at a single point.  We claim that $W'_d$ is smooth at the point $((d-1)x^{d-2}y, y^{d-1})\in W'_d$.

To show this, change coordinates via $a_1 \mapsto a_1 +1$ and $b_{d-1} \mapsto b_{d-1} +1$, and let $I$ be the ideal  generated by the $4 \times 4$ minors of the matrix
\begin{equation}
\begin{pmatrix}
a_0 & a_1+1 & a_2 & \cdots & a_{d-2} \\
a_1+1 & a_2 & a_3 & \cdots & a_{d-1} \\
b_0 & b_1 & b_2 & \cdots & b_{d-2} \\
b_1 & b_2 & b_3 & \cdots & b_{d-1} +1
\end{pmatrix}.\label{matrixhere}
\end{equation}
If we set 
$
\fp := \bigl(a_0, \ldots,a_{d-1}, b_0, \ldots, b_{d-1}\bigr),  
$
we need to prove that the localization $(\CC[a_0, \ldots,a_{d-1}, b_0, \ldots, b_{d-1}]/I)_{\fp}$ is a regular local ring.  For this, it suffices to show $\dim
\fp/\fp^2 \le \dim W'_d = d+4$.
Now, the minor of \eqref{matrixhere} given by the columns $0,1,i$ and $d-2$ for $i=2, \ldots, d-3$ is equal to $-b_i$ plus higher-order terms. Therefore, the vector space $\fp/\fp^2$ is spanned by $a_0, \ldots, a_{d-1}$, $b_0$, $b_1$, $b_{d-2}$, $b_{d-1}$, hence $\dim
\fp / \fp^2 \le d+4$ as required.

Since $W'_d$ is reduced, we have $W_d = W'_d$, which yields that $W_d$ is Cohen-Macaulay.
The proposition will therefore follow from Serre's criterion if we verify that $W_d$ is smooth away from a subvariety of codimension at least $2$.  By Proposition \ref{P:generically-closed}, we know that $W_d\cap\pi^{-1}(\Gr(2, \oh(\CC^2)_{d-1})^{\ss})$ is smooth at all points except possibly 
along the preimage $\Lambda$ of $\pi_2^{-1} (\pi_2 ( \langle x^{d-1}, y^{d-1} \rangle))$ in $\oh(\CC^2)_{d-1}^{\oplus 2}$.   By Remark \ref{R:codim1}, $\bar{\Lambda} \subseteq W_d$ is an irreducible subvariety of codimension $1$.  Since the point  $((d-1)x^{d-2}y, y^{d-1})$ is both in $\bar{\Lambda}$ and a smooth point of $W_d$, it follows that $W_d \cap\pi^{-1}(\Gr(2, \oh(\CC^2)_{d-1})^{\ss})$ is smooth in codimension $1$.  On the other hand, it is easy to see that $\dim \pi^{-1} (\Gr(2, \oh(\CC^2)_{d-1}) \setminus \Gr(2, \oh(\CC^2)_{d-1})^{\ss}) \le \dim W_d - 2$ by appealing to Proposition \ref{P:geometric-characterization}. Also, the dimension of the the locus $\{ (f_1, f_2) \mid f_1, f_2 \text{ are linearly dependent}\} \subseteq \oh(\CC^2)_{d-1}^{\oplus 2}$ clearly does not exceed $\dim W_d -2$.  Therefore, we conclude that $W_d$ is smooth in codimension $1$.
\end{proof}

\begin{cor} \label{C:normal} \it
The image 
$\bar{\nabla} \big( \PP(\oh(\CC^2)_d)^{\ss} \gitq \SL_2 \big) \subseteq  \Gr(2, \oh(\CC^2)_{d-1})^{\ss} \gitq \SL_2$
is normal.
\end{cor}

\begin{proof} Indeed, $\bar{\nabla} \big( \PP(\oh(\CC^2)_d)^{\ss} \gitq \SL_2 \big)$ is normal as the GIT quotient of the normal variety $W_d\cap\pi^{-1}(\Gr(2, \oh(\CC^2)_{d-1})^{\ss})$.
\end{proof}

\subsection{Proof of Theorem \ref{T:binary}}
 
We now establish:
 
\begin{proposition} \label{C:closed}  The morphism $\bar{\nabla} \co \PP(\oh(\CC^2)_d)^{\ss} \gitq \SL_2 \to \Gr(2, \oh(\CC^2)_{d-1})^{\ss} \gitq \SL_2$ is a closed immersion.
 \end{proposition}

\begin{proof}
The morphism $\bar{\nabla}$ is injective by Corollary \ref{C:finite-injective}. Since the image is normal by Corollary \ref{C:normal}, Zariski's Main Theorem implies that $\bar{\nabla}$ is an isomorphism onto its image.
\end{proof}

\begin{proof}[Proof of Theorem {\rm \ref{T:binary}}]
This follows by combining Corollary \ref{C:equiv-conj} with Proposition \ref{C:closed}. 
 \end{proof}

\end{document}